\documentclass{amsart}

\makeatletter
\@namedef{subjclassname@2020}{%
  \textup{2020} Mathematics Subject Classification}
 
\makeatother 

\usepackage{amsthm,amssymb,amsfonts,latexsym,mathtools,thmtools,amsmath,lscape}
\usepackage[T1]{fontenc}
\usepackage{tikz-cd} 
\usepackage{enumitem} 
\usepackage{longtable}
\usepackage{multirow}
\usepackage{mathrsfs} 
\usepackage{hyperref} 
\usepackage{hyperref}
\usepackage[all]{xy}
\usepackage{booktabs}
\hypersetup{
    colorlinks=true,
    linkcolor=blue,
    filecolor=blue,      
    urlcolor=blue,
    linktocpage=true
}

\newtheorem{theorem}{Theorem}[section]
\newtheorem{lemma}[theorem]{Lemma}
\newtheorem{proposition}[theorem]{Proposition}

\theoremstyle{definition}
\newtheorem{definition}[theorem]{Definition}
\newtheorem{corollary}[theorem]{Corollary}
\newtheorem{remark}[theorem]{Remark}

\newtheorem{example}[theorem]{Example}

\theoremstyle{remark}

\numberwithin{equation}{section}

\begin{document}

\title[Smooth geometry of graded skew Clifford algebras]{Smooth geometry of graded skew Clifford algebras}


\author{Andr\'es Rubiano}
\address{Universidad ECCI}
\curraddr{Campus Universitario}
\email{arubianos@ecci.edu.co}
\thanks{}


\thanks{}



\date{}

\dedicatory{Dedicated to Jorge Camacho}

\begin{abstract} 

In this paper, we investigate the differential smoothness of graded skew Clifford algebras.

\end{abstract}

\maketitle


\section{Introduction}

The theory of connections in noncommutative geometry is by now classical (see, for instance, the expositions by Connes \cite{Connes1994} and by Giachetta et al.~\cite{Giachettaetal2005}). One starts with a differential graded algebra $\Omega A = \bigoplus_{n\geq 0} \Omega^n A$ over a $\Bbbk$-algebra $A = \Omega^0 A$, where $\Bbbk$ is a field, and defines a {\em connection} on a left $A$-module $M$ to be a linear map
\[
\nabla^0 : M \longrightarrow \Omega^1 A \otimes_A M
\]
satisfying the Leibniz rule
\[
\nabla^0(am) = da \otimes_A m + a\,\nabla^0(m), \qquad a\in A,\ m\in M.
\]
This framework is obtained by replacing, in the classical geometric picture, commutative algebras of functions on a manifold $X$ and modules of sections of vector bundles over $X$ by noncommutative algebras and one-sided modules. As Brzezi\'nski emphasizes, this formulation \textquotedblleft captures very well the classical context in which connections appear and transports it successfully to the realm of noncommutative geometry\textquotedblright\ \cite[p.~557]{Brzezinski2008}.

In \cite{Brzezinski2008}, Brzezi\'nski observes that, from a purely algebraic standpoint, this notion of connection reflects only part of a more symmetric picture. On the one hand, noncommutative connections are defined using the tensor functor, while the latter admits the {\em hom-functor} as a right adjoint; it is therefore natural to ask whether there exist connection-type objects formulated in terms of hom-functors. On the other hand, the linear dual $M^*$ of $M$ is a right $A$-module, but a left connection on $M$ in the above sense does not automatically induce a right connection on $M^*$. Taking into account the adjointness between tensor and hom, one is led to expect that any such induced structure must involve hom-functors in an essential way.

Guided by these considerations, Brzezi\'nski develops in \cite{Brzezinski2008} a natural and rich theory of connection-like maps acting on spaces of module homomorphisms. Because of the central role of $\operatorname{Hom}$-spaces, these objects are called {\em hom-connections} (they are also referred to as {\em divergences}, since, when $A$ is an algebra of functions on $\mathbb{R}^n$ and $\Omega^1(A)$ is the standard module of one-forms, one recovers the classical divergence operator of vector calculus \cite[p.~892]{Brzezinski2011}). Brzezi\'nski shows that hom-connections arise canonically from (strong) connections on {\em noncommutative principal bundles}, and that every left connection on a bimodule in the sense of Cuntz and Quillen \cite{CuntzQuillen1995} yields a hom-connection. Furthermore, he studies induction procedures for hom-connections via differentiable bimodules (and hence via morphisms of differential graded algebras), proves that any hom-connection extends to {\em higher forms}, and introduces a notion of {\em curvature}. Iterating a hom-connection can then be expressed in terms of this curvature, leading to a chain complex attached to a {\em flat} (curvature-zero) hom-connection; the homology of this complex can be viewed as dual to the twisted cohomology associated to an ordinary connection, a perspective that plays an important role in the theory of noncommutative differential fibrations \cite{BeggsBrzezinski2005}.

Subsequently, Brzezi\'nski, El Kaoutit and Lomp \cite{BrzezinskiElKaoutitLomp2010} constructed classes of {\em differential calculi} that admit hom-connections. Their construction is based on {\em twisted multi-derivations} and yields first-order calculi $\Omega^1(A)$ which are free as left and right $A$-modules. In geometric terms, $\Omega^1(A)$ is interpreted as a module of sections of the cotangent bundle over the ``noncommutative manifold'' represented by $A$, so this construction models parallelizable manifolds or coordinate charts. Later, Brzezi\'nski remarked that \textquotedblleft one should expect $\Omega^1(A)$ to be a finitely generated and projective module over $A$ (thus corresponding to sections of a non-trivial vector bundle by the Serre–Swan theorem)\textquotedblright\ \cite[p.~885]{Brzezinski2011}, and he extended the approach of \cite{BrzezinskiElKaoutitLomp2010} to the case where $\Omega^1(A)$ is finitely generated and projective.


Building on these ideas, Brzezi\'nski and Sitarz \cite{BrzezinskiSitarz2017} introduced a different, more differential-geometric notion of smoothness, called {\em differential smoothness}. Here, the focus is on differential graded algebras of a fixed dimension that admit a noncommutative analogue of the Hodge star isomorphism: one requires the existence of a top-degree form in a differential calculus over the algebra together with a version of Poincar\'e duality implemented as an isomorphism between complexes of differential and integral forms. This differential smoothness is conceptually distinct from homological smoothness and is more constructive in nature. As they explain, \textquotedblleft the idea behind the {\em differential smoothness} of algebras is rooted in the observation that a classical smooth orientable manifold, in addition to the de Rham complex of differential forms, also carries a complex of {\em integral forms} isomorphic to the de Rham complex \cite[Section~4.5]{Manin1997}. The de Rham differential can be seen as a special left connection, while the boundary operator in the complex of integral forms is an example of a {\em right connection}\textquotedblright\ \cite[p.~413]{BrzezinskiSitarz2017}.

Differential smoothness has since been established for a wide range of noncommutative algebras, including quantum two- and three-spheres, the quantum disc and plane, the noncommutative torus, the coordinate algebras of the quantum group $SU_q(2)$, the noncommutative pillow algebra, quantum cone algebras, quantum polynomial algebras, Hopf algebra domains of Gelfand–Kirillov dimension two that are not PI, various Ore extensions, certain 3-dimensional skew polynomial algebras, diffusion algebras in three generators, and noncommutative coordinate algebras arising as deformations of classical orbifolds such as the pillow orbifold, singular cones and lens spaces (see, for example, \cite{Brzezinski2015, Brzezinski2016, BrzezinskiElKaoutitLomp2010, BrzezinskiLomp2018, BrzezinskiSitarz2017, DuboisViolette1988, DuboisVioletteKernerMadore1990, Karacuha2015, KaracuhaLomp2014, ReyesSarmiento2022}). Notably, several of these algebras are also homologically smooth in the sense of Van den Bergh, highlighting interesting interactions between the homological and differential approaches to smoothness.

The origins of Clifford algebras can be traced back to 1876, when the philosopher and geometer
William Kingdon Clifford (1845–1879), then professor of applied mathematics at University College
London, wrote the manuscript entitled \emph{On the Classification of Geometric Algebras}. This text,
found among Clifford’s notes, corresponds to an abstract intended for presentation at the London
Mathematical Society on March 10, 1876, but remained unfinished due to his deteriorating health.
Clifford, who suffered from tuberculosis, passed away on March 3, 1879, at the age of thirty–five.

From a mathematical standpoint, Clifford’s contribution builds on Grassmann’s notion of product of
\emph{extensive quantities}: the product of two such quantities $u$ and $v$ is allowed to be either
another extensive quantity or a scalar. Starting from this idea, Clifford considered $n$ fundamental
units $e_1,\dots,e_n$ satisfying $e_j^2 = \pm 1$ and $e_j e_k = - e_k e_j$ for $j \neq k$, and studied
the associative algebra generated by these elements and their products. This leads to the now
classical notion of Clifford algebra. Formally, let $V$ be a vector space over a field $\Bbbk$ and
let $q : V \to \Bbbk$ be a quadratic form. The \emph{Clifford algebra} $C(q)$ is an associative
unital $\Bbbk$–algebra equipped with a linear map $i : V \to C(q)$ with the following universal
property: for every associative unital $\Bbbk$–algebra $A$ and every linear map
$j : V \to A$ such that $j(v)^2 = q(v)\,1_A$ for all $v \in V$, there exists a unique algebra
homomorphism $f : C(q) \to A$ satisfying $f \circ i = j$. This construction has been extensively
studied and plays a fundamental role in mathematical physics, for instance in the description of
spin-$\tfrac{1}{2}$ particles, the Dirac operator, Maxwell’s equations and the Dirac equation.

More refined versions arise by incorporating grading and noncommutativity. Using a suitable
$\mathbb{Z}$–grading, one obtains \emph{graded Clifford algebras}, which already occupy an important
place in noncommutative algebra, where they are related to other classes such as PBW extensions and
Ore extensions. A further level of generality is achieved by allowing skew–commutation relations in
the generators; this leads to the notion of \emph{graded skew Clifford algebras}.

Motivated by the ongoing development of differential smoothness, the aim of this paper is to investigate this property within the framework of graded Clifford and graded skew Clifford algebras. These algebras form a rich and highly structured class of quadratic Artin-Schelter regular algebras, playing a central role in the noncommutative projective geometry of quadric hypersurfaces and in the classification of regular algebras of low global dimension \cite{Cassidy2010, Cassidy2014, Vancliff2015b}. In particular, graded skew Clifford algebras, as introduced and developed in \cite{Cassidy2010, Nafari2015, Veerapen2017}, provide a flexible deformation framework that captures many families arising from noncommutative quadrics, elliptic curves and quantum $\mathbb{P}^n$. Understanding whether such algebras admit integrable differential calculi in the sense of Brzezi\'nski and Sitarz \cite{BrzezinskiSitarz2017} is a natural and, to date, largely unexplored question.

The article is organized as follows. In Section \ref{PreliminariesDifferentialsmoothnessofbi-quadraticalgebras} we recall the necessary preliminaries on differential graded algebras, hom-connections and the notion of differential smoothness in the sense of Brzezi\'nski and Sitarz, emphasizing the role of integrable calculi and volume forms. In Section \ref{Differentialandintegralcalculusbi-quadraticalgebras} we introduce graded Clifford and graded skew Clifford algebras, describe the associated first-order differential calculi arising from suitable families of automorphisms and matrices, and establish our main criterion for differential smoothness of these algebras. Section \ref{examples} is devoted to applications of this criterion: we analyze several concrete families of graded (skew) Clifford algebras, showing in each case whether the hypotheses of the main theorem are satisfied, and we include an example that falls outside this framework. Finally, in Section \ref{FutureworkDifferentialsmoothness} we outline some directions for future research, including possible extensions of our methods to higher-dimensional families and to more general classes of noncommutative graded algebras.

Throughout the paper, $\mathbb{N}$ denotes the set of natural numbers including zero. The word ring means an associative ring with identity not necessarily commutative. $Z(R)$ denotes the center of the ring $R$. All vector spaces and algebras (always associative and with unit) are over a fixed field $\Bbbk$. $\Bbbk^{*}$ denotes the non-zero elements of $\Bbbk$. As usual, the symbols $\mathbb{R}$ and $\mathbb{C}$ denote the fields of real and complex numbers, respectively. 

\section{Preliminaries}\label{PreliminariesDifferentialsmoothnessofbi-quadraticalgebras}

We start by recalling the preliminaries on differential smoothness of algebras and Clifford algebras that are necessary for the rest of the paper.

\subsection{Differential smoothness of algebras}\label{DefinitionsandpreliminariesDSA}




We adopt the framework for differential smoothness developed by Brzezi\'nski and Sitarz in \cite[Section~2]{BrzezinskiSitarz2017} (see also \cite{Brzezinski2008, Brzezinski2014} for background).

\begin{definition}[{\cite[Section 2.1]{BrzezinskiSitarz2017}}]\label{defDGA}
\begin{enumerate}
    \item[\rm (i)] A {\em differential graded algebra} is a non-negatively graded algebra
    \[
       \Omega = \bigoplus_{n\ge 0}\Omega^n
    \]
    with product denoted by $\wedge$, together with a degree-one linear map
    \[
       d:\Omega^{\bullet} \longrightarrow \Omega^{\bullet +1}
    \]
    such that $d\circ d=0$ and, for all homogeneous $a\in \Omega^{\bullet}$ and all $b\in \Omega$,
    \[
      d(ab)=(da)b+(-1)^{\bullet}a\,db,
    \]
    that is, $d$ satisfies the graded Leibniz rule.

    \item[\rm (ii)] A differential graded algebra $(\Omega(A), d)$ is called a {\em calculus over an algebra} $A$ if $\Omega^0 (A) = A$ and, for each $n\in \mathbb{N}$,
    \[
       \Omega^n (A) = A\, dA \wedge dA \wedge \dotsb \wedge dA,
    \]
    where $dA$ appears $n$ times. This requirement is referred to as the {\em density condition}. We write $\Omega (A) = \bigoplus_{n\in \mathbb{N}} \Omega^{n}(A)$. By iterating the Leibniz rule, one also obtains
    \[
       \Omega^n (A) = dA \wedge dA \wedge \dotsb \wedge dA\, A.
    \]
    A differential calculus $\Omega (A)$ is said to be {\em connected} if ${\rm ker}\big(d\mid_{\Omega^0 (A)}\big) = \Bbbk$.

    \item[\rm (iii)] A calculus $(\Omega (A), d)$ is said to have {\em dimension} $n$ if $\Omega^n (A)\neq 0$ and $\Omega^m (A) = 0$ for all $m > n$. An $n$-dimensional calculus $\Omega (A)$ is said to {\em admit a volume form} if $\Omega^n (A)$ is isomorphic to $A$ as both a left and a right $A$-module.
\end{enumerate}
\end{definition}

\begin{remark}
The product $\wedge$ on $\Omega(A)$ should not be confused with the usual exterior (antisymmetric) wedge product; it is simply the notation used for the multiplication in the differential graded algebra.
\end{remark}

The existence of a right $A$-module isomorphism $\Omega^n (A)\cong A$ means that there is a distinguished element $\omega\in \Omega^n(A)$ which is a free generator of $\Omega^n (A)$ as a right $A$-module, so that every element of $\Omega^n(A)$ can be written uniquely as $\omega a$ with $a\in A$. If, in addition, $\omega$ is a free generator of $\Omega^n (A)$ as a left $A$-module, then $\omega$ is called a {\em volume form} on $\Omega(A)$.

The right $A$-module isomorphism $\Omega^n (A) \to A$ associated with a volume form $\omega$ will be denoted by $\pi_{\omega}$, that is,
\begin{equation}\label{BrzezinskiSitarz2017(2.1)}
\pi_{\omega} (\omega a) = a, \qquad \text{for all } a\in A.
\end{equation}

Using that $\Omega^n (A)$ is also isomorphic to $A$ as a left $A$-module, any free generator $\omega$ determines an algebra endomorphism $\nu_{\omega}$ of $A$ via
\begin{equation}\label{BrzezinskiSitarz2017(2.2)}
    a \omega = \omega \nu_{\omega} (a), \qquad a\in A.
\end{equation}
If $\omega$ is a volume form, then $\nu_{\omega}$ is in fact an algebra automorphism.

We now recall the basic ingredients of the {\em integral calculus} on $A$, which may be viewed as dual to its differential calculus (see \cite{Brzezinski2008, BrzezinskiElKaoutitLomp2010} for details).

Let $(\Omega (A), d)$ be a differential calculus on $A$. Each space of $n$-forms $\Omega^n (A)$ is naturally an $A$-bimodule. Denote by $\mathcal{I}_{n}A$ the right dual of $\Omega^{n}(A)$, namely the space of all right $A$-linear maps $\Omega^{n}(A)\rightarrow A$,
\[
\mathcal{I}_{n}A := {\rm Hom}_{A}(\Omega^{n}(A),A).
\]
Each $\mathcal{I}_{n}A$ is an $A$-bimodule with left and right actions given by
\[
    (a\cdot\phi\cdot b)(\omega)=a\,\phi(b\omega),\quad \text{for all } \phi \in \mathcal{I}_{n}A,\ \omega \in \Omega^{n}(A),\ a,b \in A.
\]

The direct sum $\mathcal{I}A = \bigoplus_{n} \mathcal{I}_n A$ carries a natural right $\Omega (A)$-module structure defined by
\begin{align}\label{BrzezinskiSitarz2017(2.3)}
    (\phi\cdot\omega)(\omega')=\phi(\omega\wedge\omega'),\quad \text{for all } \phi\in\mathcal{I}_{n + m}A,\ \omega\in \Omega^{n}(A),\ \omega' \in \Omega^{m}(A).
\end{align}

\begin{definition}[{\cite[Definition 2.1]{Brzezinski2008}}]
A {\em divergence} (or {\em hom-connection}) on $A$ is a linear map $\nabla: \mathcal{I}_1 A \to A$ such that
\begin{equation}\label{BrzezinskiSitarz2017(2.4)}
    \nabla(\phi \cdot a) = \nabla(\phi)\, a + \phi(da), \quad \text{for all } \phi \in \mathcal{I}_1 A,\ a \in A.
\end{equation}  
\end{definition}

Such a divergence extends canonically to a family of maps
\[
\nabla_n: \mathcal{I}_{n+1} A \to \mathcal{I}_{n} A,\qquad n\ge 0,
\]
defined by
\begin{equation}\label{BrzezinskiSitarz2017(2.5)}
\nabla_n(\phi)(\omega) = \nabla(\phi \cdot \omega) + (-1)^{n+1} \phi(d \omega), \quad \text{for all } \phi \in \mathcal{I}_{n+1}(A),\ \omega \in \Omega^n (A).
\end{equation}

Combining \eqref{BrzezinskiSitarz2017(2.4)} and \eqref{BrzezinskiSitarz2017(2.5)} yields the graded Leibniz rule
\[
    \nabla_n(\phi \cdot \omega) = \nabla_{m + n}(\phi) \cdot \omega + (-1)^{m + n} \phi \cdot d\omega,
\]
for all $\phi \in \mathcal{I}_{m + n + 1} A$ and $\omega \in \Omega^m (A)$ \cite[Lemma~3.2]{Brzezinski2008}. In particular, for $n=0$, if we identify ${\rm Hom}_A(A, M)$ canonically with $M$, the map $\nabla_0$ reduces to the classical Leibniz rule.



The graded module $\mathcal{I} A$ together with the maps $\nabla_n$ forms a chain complex, called the
{\em complex of integral forms} over $A$. The cokernel of $\nabla$ gives rise to a canonical map
\[
\Lambda: A \longrightarrow {\rm Coker}\,\nabla \;=\; A / {\rm Im}\,\nabla,
\]
which is referred to as the {\em integral on $A$ associated to} $\mathcal{I}A$.

Given a left $A$-module $X$ with action $a\cdot x$ for $a\in A$ and $x\in X$, and an algebra automorphism
$\nu$ of $A$, we denote by ${}^{\nu}X$ the $A$-module with the same underlying vector space $X$ but with
twisted left action
\[
  a\cdot x \;:=\; \nu(a)\,x,\qquad a\in A,\ x\in X.
\]

The next notion captures, in the noncommutative setting, a Hodge-type duality between the de~Rham complex
of differential forms and a dual complex of integral forms; see \cite[p.~112]{Brzezinski2015}.

\begin{definition}[{\cite[Definition 2.1]{BrzezinskiSitarz2017}}]
An $n$-dimensional differential calculus $(\Omega (A), d)$ is called {\em integrable} if there exists a
complex of integral forms $(\mathcal{I}A, \nabla)$ on $A$, an algebra automorphism $\nu$ of $A$, and
$A$-bimodule isomorphisms
\[
\Theta_k: \Omega^{k} (A) \longrightarrow {}^{\nu} \mathcal{I}_{n-k}A,\qquad k = 0, \dotsc, n,
\]
such that the following diagram commutes:
\[
\begin{tikzcd}
A \arrow{r}{d} \arrow{d}{\Theta_0} 
  & \Omega^{1} (A) \arrow{d}{\Theta_1} \arrow{r}{d} 
    & \Omega^2 (A)  \arrow{d}{\Theta_2} \arrow{r}{d} 
      & \dotsb \arrow{r}{d} 
        & \Omega^{n-1} (A) \arrow{d}{\Theta_{n-1}} \arrow{r}{d} 
          & \Omega^n (A)  \arrow{d}{\Theta_n} \\
{}^{\nu} \mathcal{I}_n A \arrow[swap]{r}{\nabla_{n-1}} 
  & {}^{\nu} \mathcal{I}_{n-1} A \arrow[swap]{r}{\nabla_{n-2}} 
    & {}^{\nu} \mathcal{I}_{n-2} A \arrow[swap]{r}{\nabla_{n-3}} 
      & \dotsb \arrow[swap]{r}{\nabla_{1}} 
        & {}^{\nu} \mathcal{I}_{1} A \arrow[swap]{r}{\nabla} 
          & {}^{\nu} A
\end{tikzcd}
\]
The element
\[
\omega := \Theta_n^{-1}(1) \in \Omega^n (A)
\]
is called an {\em integrating volume form}.
\end{definition}

Classical examples of algebras admitting integrable calculi include the algebra of complex matrices $M_n(\mathbb{C})$ with the $n$-dimensional calculus generated by derivations constructed by Dubois-Violette et al.~\cite{DuboisViolette1988, DuboisVioletteKernerMadore1990}, the quantum group $SU_q(2)$ equipped with Woronowicz’s three-dimensional left-covariant calculus \cite{Woronowicz1987} and the corresponding restriction to the quantum standard sphere. Further examples and a systematic discussion can be found in \cite{BrzezinskiElKaoutitLomp2010}.

The most interesting cases of differential calculi are those where $\Omega^k (A)$ are finitely generated and projective right or left (or both) $A$-modules \cite{Brzezinski2011}.

\begin{proposition}\label{BrzezinskiSitarz2017Lemmas2.6and2.7}
\begin{enumerate}
\item [\rm (1)] \cite[Lemma 2.6]{BrzezinskiSitarz2017} Consider $(\Omega (A), d)$ an integrable and $n$-dimensional calculus over $A$ with integrating form $\omega$. Then $\Omega^{k} (A)$ is a finitely generated projective right $A$-module if there exist a finite number of forms $\omega_i \in \Omega^{k} (A)$ and $\overline{\omega}_i \in \Omega^{n-k} (A)$ such that, for all $\omega' \in \Omega^{k} (A)$, we have that 
\begin{equation*}
\omega' = \sum_{i} \omega_i \pi_{\omega} (\overline{\omega}_i \wedge \omega').
\end{equation*}

\item [\rm (2)] \cite[Lemma 2.7]{BrzezinskiSitarz2017} Let $(\Omega (A), d)$ be an $n$-dimensional calculus over $A$ admitting a volume form $\omega$. Assume that for all $k = 1, \ldots, n-1$, there exists a finite number of forms $\omega_{i}^{k},\overline{\omega}_{i}^{k} \in \Omega^{k}(A)$ such that for all $\omega'\in \Omega^k(A)$, we have that
\begin{equation*}
\omega'=\displaystyle\sum_i\omega_{i}^{k}\pi_\omega(\overline{\omega}_{i}^{n-k}\wedge\omega')=\displaystyle\sum_i\nu_{\omega}^{-1}(\pi_\omega(\omega'\wedge\omega_{i}^{n-k}))\overline{\omega}_{i}^{k},
\end{equation*}

where $\pi_{\omega}$ and $\nu_{\omega}$ are defined by {\rm (}\ref{BrzezinskiSitarz2017(2.1)}{\rm )} and {\rm (}\ref{BrzezinskiSitarz2017(2.2)}{\rm )}, respectively. Then $\omega$ is an integral form and all the $\Omega^{k}(A)$ are finitely generated and projective as left and right $A$-modules.
\end{enumerate}
\end{proposition}

Brzezi\'nski and Sitarz \cite[p.~421]{BrzezinskiSitarz2017} pointed out that, in order to relate the integrability of a differential graded algebra $(\Omega A, d)$ to the underlying algebra $A$, one needs to compare the dimension of the differential calculus with an appropriate notion of dimension for $A$. Since the algebras under consideration often arise as deformations of coordinate rings of affine varieties, the most natural invariant to use is the {\em Gelfand-Kirillov dimension}, introduced by Gelfand and Kirillov in \cite{GelfandKirillov1966, GelfandKirillov1966b}. Concretely, if $A$ is an affine $\Bbbk$-algebra, its {\em Gelfand--Kirillov dimension}, denoted ${\rm GKdim}(A)$, is defined by
\[
{\rm GKdim}(A)\;:=\;\underset{n\to \infty}{\limsup}\;
\frac{\log\bigl(\dim_{\Bbbk} V^{n}\bigr)}{\log n},
\]
where $V$ is any finite-dimensional subspace of $A$ that generates $A$ as an algebra. This quantity does not depend on the particular choice of $V$. For a non-affine algebra $A$, one sets ${\rm GKdim}(A)$ to be the supremum of the Gelfand--Kirillov dimensions of all affine subalgebras of $A$. An affine domain of Gelfand--Kirillov dimension zero is exactly a division ring that is finite-dimensional over its center.

Similarly, an affine domain of Gelfand--Kirillov dimension one over $\Bbbk$ is a finite module over its center and hence satisfies a polynomial identity. In this sense, ${\rm GKdim}$ measures how far $A$ is from being finite-dimensional. We refer the reader to the monograph of Krause and Lenagan
\cite{KrauseLenagan2000} for a detailed account of this invariant.

With these preliminaries in place, we recall the central notion of this work.

\begin{definition}[{\cite[Definition 2.4]{BrzezinskiSitarz2017}}]\label{BrzezinskiSitarz2017Definition2.4}
Let $A$ be an affine algebra with integer Gelfand--Kirillov dimension ${\rm GKdim}(A)=n$. We say that $A$
is {\em differentially smooth} if there exists an $n$-dimensional, connected, integrable differential
calculus $(\Omega(A), d)$ over $A$.
\end{definition}

Thus, by Definition~\ref{BrzezinskiSitarz2017Definition2.4}, a differentially smooth algebra is endowed not only with a well-behaved noncommutative differential structure, but also with a canonical notion of integration compatible with this calculus \cite[p.~2414]{BrzezinskiLomp2018}.

\begin{example}
As already mentioned in the Introduction, many noncommutative algebras have been shown to be differentially smooth (see, for instance, \cite{Brzezinski2015, BrzezinskiElKaoutitLomp2010, BrzezinskiLomp2018, BrzezinskiSitarz2017, Karacuha2015, KaracuhaLomp2014, ReyesSarmiento2022}). A basic example is the polynomial algebra $\Bbbk[x_1,\dotsc,x_n]$, which has Gelfand--Kirillov dimension $n$, and whose standard exterior calculus furnishes an $n$-dimensional integrable differential calculus; hence $\Bbbk[x_1,\dotsc,x_n]$ is differentially smooth. Furthermore, it is known from \cite{BrzezinskiElKaoutitLomp2010} that the coordinate algebras of the quantum group $SU_q(2)$, the standard quantum Podle\'s sphere and the quantum Manin plane are also differentially smooth.
\end{example}

\begin{remark}
As one might expect, there are also natural examples of algebras which fail to be differentially smooth. Consider, for instance, the commutative algebra $A=\mathbb{C}[x,y]/\langle xy\rangle$. A contradiction argument shows that there is no one-dimensional connected integrable calculus over $A$, and therefore $A$ cannot be differentially smooth \cite[Example 2.5]{BrzezinskiSitarz2017}.
\end{remark}

\subsection{Clifford algebras}\label{BiquadraticalgebrasPBWbasis}

In this part, the basic notions for defining graded Clifford algebras will be given, along with some important properties.

\begin{definition}[{\cite[Definition 2.1]{Veerapen2017}}]\label{D1110}
Let $M_{1},\ldots,M_{n} \in M_{n}(\Bbbk )$ be symmetric matrices. A \textit{graded Clifford algebra} is a $\Bbbk$-algebra $C$ on degree-one generators $x_{1},\ldots,x_{n}$ and on degree-two generators $y_{1},\ldots,y_{n}$ with the following defining relations:

\begin{enumerate}
\item [\rm (i)] $x_{i}x_{j}+x_{j}x_{i}$ = $\sum_{k=1}^{n}\left(M_{k}\right)_{ij}y_{k}$, for all $i,j = 1,\ldots,n$;
\item [\rm (ii)] $y_{k}$ is central, for all $k = 1,...,n$.
\end{enumerate}
\end{definition}

Next, we show a simple example of an algebra that meets both conditions to be graded Clifford algebra.

\begin{example}[{\cite[Example 1]{Veerapen2017}}]\label{E1111}
Suppose that we have  $
M_{1} = \left[ \begin{array}{ll}
2 & \lambda  \\ \lambda & 0  \\
\end{array}\right]
$ and $
M_{2} = \left[ \begin{array}{ll}
0 & 0  \\ 0 & 1  \\
\end{array}\right]
$, where $\lambda \in \Bbbk $. Let $C$ be the graded $\Bbbk $-algebra on degree-one generators $x_{1}$, $x_{2}$ and on degree-two generators $y_{1}$, $y_{2}$ with definitions given by Definition \ref{D1110} (i). The defining relations of $C$ are
\[
2x_{1}^{2} = 2y_{1}, \ x_{2}^{2} = y_{2}, \text{ and } x_{1}x_{2}+x_{2}x_{1} = \lambda y_{1} = \lambda x_{1}^{2}.
\]
\end{example}

In order to define graded skew Clifford algebras, we must first have the generalized concept of symmetric matrix.

\begin{definition}[{\cite[Definition 3.2]{Veerapen2017}}]\label{D1112}
For $\{i,j\} \subset \{1,\ldots,n\}$, let $\mu _{ij} \in \Bbbk ^{\times}$ such that $\mu _{ij}\mu _{ji} = 1$, for all $i,j$ where $i \not = j$, and $\mu _{ii} = 1$, for all $i=1,\ldots, n$. We write $\mu = \left(\mu _{ij}\right) \in M_n(\Bbbk )$.

A matrix $M \in M(n,\Bbbk )$ is called $\mu$-\textit{symmetric}, if $(M)_{ij}=\mu _{ij}(M)_{ji}$, for all $i,j = 1,\ldots,n$. We write $M^{\mu}(n,\Bbbk )$ for the set of $\mu$-symmetric matrices in $M(n,\Bbbk )$.
\end{definition}

\begin{definition}[{\cite[Definition 3.2]{Veerapen2017}}]\label{D1114} A \textit{graded skew Clifford algebra} denoted by $C = C(\mu,M_{1},\ldots,M_{n})$ associated to $\mu$ and $M_{1},\ldots,M_{n} \in M^{\mu}(n,\Bbbk )$ is a $\mathbb{Z}$-graded $\Bbbk $-algebra on degree-one generator $x_{1},\ldots, x_{n}$ and on degree-two generators $y_{1},\ldots,y_{n}$ with defining relations given by:
\begin{enumerate}
\item [\rm (i)] $x_{i}x_{j}+\mu _{ij} x_{j}x_{i}$ = $\sum_{k=1}^{n}\left(M_{k}\right)_{ij}y_{k}$, for all $i,j = 1,\ldots,n$;
\item [\rm (ii)] the existence of a normalizing sequence $\{y_{1}^{'},\ldots,y_{n}^{'}\}$ that spans $\Bbbk y_{1}+ \cdots + \Bbbk y_{n}$.
\end{enumerate}

\end{definition}

\begin{example}[{\cite[Examples 2.5]{Vancliff2015a}}]\label{E1115} (Quantum affine plane) Let $n = 2$, and consider $
M_{1} = \left[ \begin{array}{cc}
2 & 0  \\ 0 & 0  \\
\end{array}\right]
$ and $
M_{2} = \left[ \begin{array}{cc}
0 & 0 \\ 0 & 2  \\
\end{array}\right]
$. The degree-2 relations of $A(\mu, M_{1}, M_{2})$ have the form
\begin{equation*}
2x_{1}^{2} = 2y_{1}, \qquad 2x_{2}^{2} = 2y_{2}, \qquad x_{1}x_{2} + \mu_{12}x_{2}x_{1} = 0, 
\end{equation*}
so that $\Bbbk \langle x_{1}, x_{2}\rangle/\langle x_{1}x_{2} + \mu_{12}x_{2}x_{1} \rangle$ is a graded skew Clifford algebra with isomorphism $M_i \mapsto 2x_i^2$, for $ i = 1, 2$.
\end{example}

\begin{lemma}[{\cite[Lemma 1.13]{Cassidy2010}}]\label{lemmaindepent}
Let $C$ be a graded skew Clifford algebra. We set $Y=\sum_{k=1}^nM_ky_k$, so the entries $Y_{ij}$ of $Y$ satisfy $Y_{ij}=\mu_{ij}Y_{ji}$ for all $1\leq i,j \leq n$. With this notation, $x_ix_j+\mu_{ij}x_jx_i=Y_{ij}$ in $C$ for all $1\leq i,j \leq n$. So, the following are equivalent:
\begin{enumerate}
    \item [\rm (i)] $y_i\in (C_1)^2$ for all $1\leq i \leq n$; 
    \item [\rm (ii)] $\text{dim}\left(\sum_{i,j=1}^n\Bbbk Y_{ij}\right)=n$;
    \item [\rm (iii)] $M_1, \ldots, M_n$ are linearly independent.
\end{enumerate}
\end{lemma}

\section{Differential and integral calculus}\label{Differentialandintegralcalculusbi-quadraticalgebras}

In this section we investigate the differential smoothness of graded skew Clifford algebras.

\begin{theorem}\label{Theo1} Let $C$ a graded skew Clifford algebra. If $C$ satisfies conditions in Lemma \ref{lemmaindepent} and
\[
 (M_l)_{ij}=0 \ \text{or} \ \mu_{lk}^2=\mu_{ki}\mu_{jk}, \ 1 \leq k,l \leq n, \ 1 \leq i<j \leq n,
\]
then $C$ is differentially smooth.
\end{theorem}
\begin{proof}
Consider the following automorphisms for $1 \leq i<j \leq n$:
\begin{align}
   \nu_{x_i}(x_i) = &\ -x_i, & \nu_{x_i}(x_j) = &\ -\mu_{ji}x_j,\quad {\rm and} \label{Auto1} \\ 
    \nu_{x_j}(x_i) = &\ -\mu_{ij}x_i, & \nu_{x_j}(x_j) = &\ -x_j. \label{Auto2} 
\end{align}
The maps $\nu_{x_i}$, $1 \leq i \leq n$ can be extended to an algebra homomorphism of $C$ if and only if the definitions of homomorphism over the generators $x_j$, $1 \leq j \leq n$ respect relations of Definition \ref{D1114}, i.e.
\begin{align*}
   \nu_{x_i}(x_i)\nu_{x_i}(x_j)+\mu_{ij}\nu_{x_i}(x_j)\nu_{x_i}(x_i)  & =  \sum_{k=1}^n\left(M_k\right)_{ij}\nu_{x_i}(x_k)\nu_{x_i}(x_k) \\
   \mu_{ji}x_ix_j+\mu_{ij}\mu_{ji}x_jx_i & = \sum_{k=1}^{n}\left(M_k\right)_{ij}\mu_{ki}^2x_k^2 \\
   x_ix_j+\mu_{ij}x_jx_i & = \sum_{k=1}^{n}\left(M_k\right)_{ij}\mu_{ij}\mu_{ki}^2x_k^2
\end{align*}
So, we have

\begin{align*}
    & \sum_{k=1}^n\left(M_k\right)_{ij}x_k^2  =\sum_{k=1}^{n}\left(M_k\right)_{ij}\mu_{ij}\mu_{ki}^2x_k^2 \\
     & \sum_{k=1}^{n}(M_k)_{ij}(1-\mu_{ij}\mu_{ki}^2)x_k^2 = 0
\end{align*}
Then,

\begin{align*}
    (M_k)_{ij}(1-\mu_{ij}\mu_{ki}^2) & = 0 \ \text{for} \ 1\leq k \leq n.
\end{align*}
With these two equations, we obtain the conditions
\begin{align}
    (M_i)_{ij}=0 \ & \text{or} \ \mu_{ij}=1, \notag \\
    (M_j)_{ij}=0 \ & \text{or} \ \mu_{ij}=1, \label{Eq1} \\
    (M_k)_{ij}=0 \ & \text{or} \ \mu_{ki}^2=\mu_{ji}, \text{ for } 1\leq k\leq n, k\ne i,j \notag.
\end{align}
Now, we do the same with $\nu_{x_j}$. 
\begin{align*}
   \nu_{x_j}(x_i)\nu_{x_j}(x_j)+\mu_{ij}\nu_{x_j}(x_j)\nu_{x_j}(x_i)  & =  \sum_{k=1}^n\left(M_k\right)_{ij}\nu_{x_j}(x_k)\nu_{x_j}(x_k) \\
   \mu_{ij}x_ix_j+\mu_{ij}\mu_{ij}x_jx_i & = \sum_{k=1}^{n}\left(M_k\right)_{ij}\mu_{kj}^2x_k^2 \\
   x_ix_j+\mu_{ij}x_jx_i & = \sum_{k=1}^{n}\left(M_k\right)_{ij}\mu_{ji}\mu_{kj}^2x_k^2.
\end{align*}
So, we have

\begin{align*}
    & \sum_{k=1}^n\left(M_k\right)_{ij}x_k^2  =\sum_{k=1}^{n}\left(M_k\right)_{ij}\mu_{ji}\mu_{kj}^2x_k^2 \\
     & \sum_{k=1}^{n}(M_k)_{ij}(1-\mu_{ji}\mu_{kj}^2)x_k^2 = 0
\end{align*}
Then,

\begin{align*}
    (M_k)_{ij}(1-\mu_{ji}\mu_{kj}^2) & = 0 \ \text{for} \ 1\leq k \leq n.
\end{align*}
The conditions obtained here are the same as those obtained with $\nu_{x_i}$.

Finally, we take $\nu_{x_k}$ when $k \ne i,j$. Here, the map $\nu_{x_k}$ can be extended to an algebra homomorphism of $C$ if and only if the equalities
\begin{align}
    (M_k)_{ij}=0 \ & \text{or} \ \mu_{ki}\mu_{jk}= 1, \label{Eq2} \\
    (M_l)_{ij}=0 \ & \text{or} \ \mu_{lk}^2=\mu_{ki}\mu_{jk}, \ \text{for} \ 1 \leq l \leq n, l \ne k \notag
\end{align}
are satisfied.

The equations \ref{Eq1} and \ref{Eq2} can be written as follows
\begin{align}
    (M_l)_{ij}=0 \ \text{or} \ \mu_{lk}^2=\mu_{ki}\mu_{jk}, \ 1 \leq k,l \leq n, \label{Eq3}
\end{align}
for $1 \leq i <j\leq n$. 
Now, since each map scales the generators by scalar multiples, the corresponding automorphisms $\nu_{x_i}$ and $\nu_{x_j}$ commute pairwise for all $1\leq i<j\leq n$. 

Let $\Omega^{1}(C)$ be the free right $C$-module of rank $n$ generated by the symbols $dx_i$, $1 \leq i \leq n$. We define a left $C$-module structure on $\Omega^{1}(C)$ by specifying, for all $p \in C$,
\begin{align}
    pdx_i = (dx_i)\nu_{x_i}(p), \ 1 \leq i \leq n \label{relrightmod}.
\end{align}

The relations in $\Omega^{1}C$ are given by 
\begin{align}
x_idx_i = &\ -(dx_i) x_i, \ 1 \leq i \leq n, \notag \\
x_idx_j = &\ -(dx_j)\mu_{ij}x_i, \ 1 \leq i<j \leq n,  \label{rel1} \\
x_jdx_i = &\ -(dx_i)\mu_{ji}x_j, \ 1 \leq i<j \leq n.  \notag
\end{align}
    
We want to extend the correspondences 
\begin{equation*}
x_i \mapsto d x_i, \ 1\leq i \leq n
\end{equation*} 

to a map $d: C \to \Omega^{1} (C)$ satisfying the Leibniz's rule. This is possible if it is compatible with the nontrivial relations of Definition \ref{D1114}, i.e. if the equalities
\begin{align*}
        dx_ix_j+x_idx_j+\mu_{ij}dx_jx_i+\mu_{ij}x_jdx_i-\sum_{k=1}^n(M_k)_{ij}\left(dx_jx_j+x_jdx_j\right)  &\ = 0
\end{align*}
hold.

Define $\Bbbk$-linear maps 
\begin{equation*}
\partial_{x_i}: C \rightarrow C, \ 1\leq i \leq n
\end{equation*}

such that
\begin{align*}
    d(a)=\sum_{i=1}^n(dx_i)\partial_{x_i}(a), \ {\rm for\ all} \ a \in C.
\end{align*}

Since $dx_i$, $1\leq i \leq n$ form a free generating set of the right $C$-module $\Omega^1(C)$, the above definitions yield well-defined left $C$-module actions. It is worth noting that $d(a) = 0$ holds if and only if $\partial_{x_i}(a) = 0$ for all $1\leq i \leq n$. Employing the relations in {\rm (}\ref{relrightmod}{\rm )} along with the definitions of the automorphisms $\nu_{x_i}$, $1\leq i \leq n$, one obtains the following expressions:
\begin{align*}
\partial_{x_i}(x_1^{l_1}\cdots x_n^{l_n}) = (-1)^{\circ}\prod_{k=1}^{i-1}\mu_{ki}^{l_k}x_1^{l_1}\cdots x_{i-1}^{l_{i-1}}x_i^{l_i-1}x_{i+1}^{l_{i+1}}\cdots x_n^{l_n},
\end{align*}
\noindent where $\circ = \sum_{j=1}^{i}l_j$, for $1\leq i \leq n$.

Consequently, $d(a) = 0$ if and only if $a$ is a scalar multiple of the identity element. This establishes that the differential graded algebra $(\Omega (C), d)$ is connected, where $\Omega (C)$ decomposes as 
\[
\Omega (C) =\bigoplus_{i=0}^{n}\Omega^i (C)
\]
Extending the differential $d$ to higher-degree forms in a manner consistent with relations {\rm (}\ref{rel1}{\rm )}, we derive the following rules for $\Omega^l(C)$  $(2\leq l\leq n-1)$:
\begin{align}
\bigwedge_{k=1}^{l}dx_{q(k)} = (-1)^{\sharp}\prod_{r,s\in P}\mu_{rs}^{-1}\bigwedge_{k=1}^ldx_{p(k)}
\end{align}

\noindent where 
$$
q:\{1,\ldots,l\}\rightarrow \{1,\ldots,n\}
$$ 

\noindent is an injective map and 
$$
p:\{1,\ldots,l\}\rightarrow \text{Im}(q)$$

\noindent is an increasing injective map and $\sharp$ is the number of $2$-permutations needed to transform $q$ into $p$, and $P := \{(s, t) \in \{1, \ldots, l\} \times \{1, \ldots, l\} \mid q(s) >  q(t)\}$.

Since the automorphisms $\nu_{x_i}$, $1\leq i \leq n$ commute with each other, there are no additional relations to the previous ones, so we get that
\begin{align*}
\Omega^{n-1} (C) = \left[\bigoplus_{r=1}^{n} \bigwedge_{\substack{j=1,\\ j \neq r}}^{n}dx_j\right]C.
\end{align*}

Now, we have $\Omega^n(C) = \omega C\cong C$ as a right and left $C$-module, with $\omega=dx_1\wedge \cdots \wedge dx_n$, where $\nu_{\omega}=\nu_{x_1}\circ\cdots\circ\nu_{x_n}$, we have that $\omega$ is a volume form of $C$. From Proposition \ref{BrzezinskiSitarz2017Lemmas2.6and2.7} (2) we get that $\omega$ is an integral form by setting
\begin{align*}
    \omega_i^j = &\ \bigwedge_{k=1}^{j}dx_{p_{i,j}(k)}, \text{ for } 1\leq i \leq \binom{n}{j-1}, \\
    \bar{\omega}_i^{n-j} = &\ (-1)^{\sharp_{i,j}}\prod_{r,s\in P_{i,j}}\mu_{r,s}^{-1}\bigwedge_{k=j+1}^{n}dx_{\bar{p}_{i,j}(k)}, \text{ for } 1\leq i \leq \binom{n}{j-1},
\end{align*}
for $1\leq j \leq n$ and where 
\begin{align*}
    p_{i,j}:\{1,\ldots,j\}\rightarrow &\ \{1,\ldots,n\}, \quad {\rm and} \\
\bar{p}_{i,j}:\{j+1,\ldots,n\}\rightarrow &\ (\text{Im}(p_{i,j}))^c
\end{align*}

(the symbol $\square^c$ denotes the complement of the set $\square$), are increasing injective maps, and $\sharp_{i,j}$ is the number of $2$-permutation needed to transform 
\[
\left\{\bar{p}_{i,j}(j+1),\ldots, \bar{p}_{i,j}(n), p_{i,j}(1), \ldots, p_{i,j}(j)\right\} \quad {\rm into\ the\ set} \quad \{1, \ldots, n\},
\]

and 
\[
P_{i,j} :=\{(s, t) \in \{1, \ldots, j\}\times\{j+1, \ldots, n\} \mid p_{i,j}(s)< \bar{p}_{i,j}(t)\}.
\]

Consider $\omega' \in \Omega^j(C)$, that is,  
\begin{align*}
\omega' =\sum_{i=1}^{\binom{n}{j-1}}\bigwedge_{k=1}^{j}dx_{p_{i,j}(k)}b_i, \quad {\rm with} \  b_i \in \Bbbk.
\end{align*}

Then
\begin{align*}
 \sum_{i=1}^{\binom{n}{j-1}}\omega_{i}^{j}\pi_{\omega}(\bar{\omega}_i^{n-j}\wedge \omega') = &\ \sum_{i=1}^{\binom{n}{j-1}}\left[\bigwedge_{k=1}^{j}dx_{p_i(k)}\right] \cdot  \pi_{\omega} \left[(-1)^{\sharp_{i,j}} \square^{*} \wedge \omega'\right] \\
 = &\ \displaystyle \sum_{i=1}^{\binom{n}{j-1}}\bigwedge_{k=1}^{j}dx_{p_{i,j}(k)}b_i =  \omega',
\end{align*}

where 
\begin{align*}
    \square^{*} := &\ \prod_{(r,s)\in P_{i,j}}\mu_{rs}^{-1} \bigwedge_{k=j+1}^{n}dx_{\bar{p}_{i,j}(k)}.
\end{align*}

By Proposition \ref{BrzezinskiSitarz2017Lemmas2.6and2.7} (2), it follows that $C$ is differentially smooth.
\end{proof}

\begin{corollary} Let $C$ a graded Clifford algebra. If $C$ satisfies conditions in Lemma \ref{lemmaindepent}, then $C$ is differentially smooth.
\end{corollary}
\begin{proof}
In this case, we have $\mu_{ij}=1$ for all $1\leq i,j \leq n$. Thus the condition $\mu_{lk}^2=\mu_{ki}\mu_{jk}, \ 1 \leq k,l \leq n$ of Theorem \ref{Theo1} is satisfied.
\end{proof}

\section{Examples}\label{examples}

In this section we present several examples of skew Clifford algebras that satisfy the hypotheses of Theorem~\ref{Theo1}, together with one example that does not, in order to illustrate both the scope and the limitations of our result.

\begin{example}
We consider Example~1 in \cite{Cassidy2010}. In this case the algebra is generated by
$x_1,x_2,x_3,x_4$ subject to the relations
\begin{align*}
     x_1x_4 + i x_4x_1 &= 0, 
     & x_3^2 &= x_1^2, 
     & x_1x_3 - x_3x_1 &= x_2^2, \\
     x_2x_3 - j^{-1}x_3x_2 &= 0, 
     & x_4^2 &= x_2^2, 
     & x_2x_4 - x_4x_2 &= \gamma x_1^2,
\end{align*}
where $i^2 = -1$, $j = \pm i$ and $\gamma \in \Bbbk^{\ast}$. The corresponding matrices
$M_1,\dots,M_4$ and $\mu$ are
\begin{align*}
   M_{1} &=
   \begin{bmatrix}
    0 & 0 & 0 & 0 \\
    0 & 0 & 0 & \gamma \\
    0 & 0 & 2 & 0 \\
    0 & -\gamma & 0 & 0 
   \end{bmatrix},
   &
   M_{2} &=
   \begin{bmatrix}
    0 & 0 & 1 & 0 \\
    0 & 0 & 0 & 0 \\
    -1 & 0 & 0 & 0 \\
    0 & 0 & 0 & 2 
   \end{bmatrix}, \\
   M_{3} &=
   \begin{bmatrix}
    2 & 0 & 0 & 0 \\
    0 & 0 & 0 & \gamma \\
    0 & 0 & 0 & 0 \\
    0 & -\gamma & 0 & 0 
   \end{bmatrix},
   &
   M_{4} &=
   \begin{bmatrix}
    0 & 0 & 1 & 0 \\
    0 & 2 & 0 & 0 \\
    -1 & 0 & 0 & 0 \\
    0 & 0 & 0 & 0 
   \end{bmatrix}, \\
   \mu &=
   \begin{bmatrix}
    1 & i & -1 & i \\
    -i & 1 & -j^{-1} & -1 \\
    -1 & -j & 1 & i \\
    -i & -1 & -i & 1 
   \end{bmatrix}.
\end{align*}
Therefore this graded skew Clifford algebra satisfies the hypotheses of
Theorem~\ref{Theo1} and is, in particular, differentially smooth.
\end{example}

\begin{example}
We now turn to Example~2 in \cite{Cassidy2010}. Here the graded skew Clifford algebra
is generated by $x_1,x_2,x_3,x_4$ with defining relations
\begin{align*}
     x_1x_3 + x_3x_1 &= \beta_2 x_2^2, 
     & x_2x_3 - x_3x_2 &= 0, 
     & x_1x_4 + x_4x_1 &= \alpha_2 x_3^2, \\
     \alpha_1 x_3^2 + \beta_1 x_2^2 &= x_1^2, 
     & x_2x_4 + x_4x_2 &= x_3^2, 
     & x_2^2 &= x_4^2,
\end{align*}
with parameters subject to
\[
\alpha_2(\alpha_2 - 1) = 0
\quad\text{and}\quad
(\alpha_1^2 + \alpha_2^2 \beta_1)(\beta_1^2 + \beta_2^2 \alpha_1) \neq 0.
\]
The corresponding matrices $M_1,\dots,M_4$ and $\mu$ are given by
\begin{align*}
   M_{1} &=
   \begin{bmatrix}
    0 & 0 & 0 & 0 \\
    0 & 0 & 0 & 0 \\
    0 & 0 & 2\alpha_1^{-1} & 0 \\
    0 & 0 & 0 & 2\beta_1^{-1}
   \end{bmatrix},
   &
   M_{2} &=
   \begin{bmatrix}
    2\beta_1 & 0 & \beta_2 & 0 \\
    0 & 0 & 0 & 0 \\
    \beta_2 & 0 & -2\alpha_1^{-1}\beta_1 & 0 \\
    0 & 0 & 0 & 0
   \end{bmatrix}, \\
   M_{3} &=
   \begin{bmatrix}
    2\alpha_1 & \pm \alpha_2 & 0 & \alpha_2 \\
    \pm \alpha_2 & 0 & 0 & 1 \\
    0 & 0 & 0 & 0 \\
    \alpha_2 & 1 & 0 & -2\alpha_1\beta_1^{-1}
   \end{bmatrix},
   &
   M_{4} &=
   \begin{bmatrix}
    0 & 0 & 0 & 0 \\
    0 & 2 & 0 & 0 \\
    0 & 0 & 0 & 0 \\
    0 & 0 & 0 & 0
   \end{bmatrix}, \\
   \mu &=
   \begin{bmatrix}
    1 & 1 & 1 & 1 \\
    1 & 1 & -1 & 1 \\
    1 & -1 & 1 & -1 \\
    1 & 1 & -1 & 1
   \end{bmatrix}.
\end{align*}
Under these hypotheses the algebra fits into the framework of
Theorem~\ref{Theo1}, and hence this graded skew Clifford algebra is
differentially smooth.
\end{example}

\begin{example}
We now examine Example~3 from \cite{Cassidy2010}. In this situation the graded skew
Clifford algebra is generated by $x_1,x_2,x_3,x_4$ with defining relations
\begin{align*}
     x_1x_3+\mu_{13}x_3x_1 &= 0, 
     & x_3x_4+\mu_{34}x_4x_3 &= 0, 
     & x_2x_3+x_3x_2 &= -x_4^2, \\
     x_1x_4+\mu_{14}x_4x_1 &= 0, 
     & x_2x_4+\mu_{24}x_4x_2 &= -x_1^2, 
     & x_2^2 &= x_4^2,
\end{align*}
where the parameters satisfy
\[
(\mu_{34})^2 = \mu_{23} = 1
\quad\text{and}\quad
\mu_{34}=\mu_{24}=(\mu_{14})^2=(\mu_{13})^2.
\]
The corresponding matrices $M_1,\dots,M_4$ and $\mu$ are
\begin{align*}
   M_{1} &=
   \begin{bmatrix}
    0 & 0 & 0 & 0 \\
    0 & 0 & 0 & -1 \\
    0 & 0 & 0 & 0 \\
    0 & -(\mu_{14})^{-2} & 0 & 0 
   \end{bmatrix},
   &
   M_{2} &=
   \begin{bmatrix}
    0 & 0 & 0 & 0 \\
    0 & 0 & 0 & 0 \\
    0 & 0 & 0 & 0 \\
    0 & 0 & 0 & 2 
   \end{bmatrix}, \\
   M_{3} &=
   \begin{bmatrix}
    0 & 0 & 0 & 0 \\
    0 & 0 & 0 & 0 \\
    0 & 0 & 0 & 0 \\
    0 & 0 & 0 & 0 
   \end{bmatrix},
   &
   M_{4} &=
   \begin{bmatrix}
    0 & 0 & 0 & 0 \\
    0 & 2 & -1 & 0 \\
    0 & -1 & 0 & 0 \\
    0 & 0 & 0 & 0 
   \end{bmatrix}, \\
   \mu &=
   \begin{bmatrix}
    1 & \mu_{14} & \pm \mu_{14} & \mu_{14} \\
    \mu_{14}^{-1} & 1 & 1 & (\mu_{14})^2 \\
    \pm (\mu_{14})^{-1} & 1 & 1 & (\mu_{14})^2 \\
    \mu_{14} & (\mu_{14})^{-2} & (\mu_{14})^{-2} & 1
   \end{bmatrix}.
\end{align*}
When $(\mu_{14})^2 = 1$ (and hence $\mu_{34}=\mu_{24}=(\mu_{13})^2=1$), the hypotheses of
Theorem~\ref{Theo1} are satisfied, and we conclude that this graded skew Clifford
algebra is differentially smooth.
\end{example}

\begin{example}[{\cite[Example 2.1]{Nafari2015}}]\label{ExampleNoDS}
Consider the graded Clifford algebra generated by $x_1,x_2,x_3$ with defining relations
\[
x_1x_2+2x_2x_1 = x_3^2,\qquad
x_1x_3+x_3x_1 = 0,\qquad
x_2x_3+x_3x_2 = 0.
\]
The associated matrices $M_1,M_2,M_3$ and $\mu$ are given by
\begin{align*}
   M_{1} &=
   \begin{bmatrix}
    2 & 0 & 0 \\
    0 & 0 & 0 \\
    0 & 0 & 0 
   \end{bmatrix},
   &
   M_{2} &=
   \begin{bmatrix}
    0 & 0 & 0 \\
    0 & 2 & 0 \\
    0 & 0 & 0
   \end{bmatrix}, \\[0.5em]
   M_{3} &=
   \begin{bmatrix}
    0 & 1 & 0 \\
    \tfrac{1}{2} & 0 & 0 \\
    0 & 0 & 2
   \end{bmatrix},
   &
   \mu &=
   \begin{bmatrix}
    1 & 2 & 1 \\
    \tfrac{1}{2} & 1 & 1 \\
    1 & 1 & 1
   \end{bmatrix}.
\end{align*}
In this case we have $(M_3)_{21}\neq 0$ and $(\mu_{31})^2 \neq \mu_{12}\mu_{11}$, so the hypotheses of Theorem~\ref{Theo1} are not satisfied. Hence this graded skew Clifford algebra lies outside the scope of our differential smoothness criterion.
\end{example}

\begin{remark}
In \cite{Nafari2011}, Nafari, Vancliff and Zhang show that certain algebras of type $A$ whose point scheme is an elliptic curve can be realized as $\Bbbk$-algebras generated by $x,y,z$ with defining relations
\[
axy+byx+cz^2= 0,\qquad ayz+bzy+cx^2= 0,\qquad azx+bxz+cy^2 = 0,
\]
where $a,b,c \in \Bbbk$, and that, under suitable conditions, these algebras are regular graded skew Clifford algebras. These algebras coincide with the $3$-dimensional Sklyanin algebras \cite{IyuduShkarin2017}. Recently, the author, Herrera and Higuera proved that all non-degenerate $3$-dimensional Sklyanin algebras are differentially smooth \cite{Herrera2025}.

\end{remark}

\section{Future work}\label{FutureworkDifferentialsmoothness}

As expected, a natural next step is to understand how far the differential smoothness criterion in Theorem~\ref{Theo1} extends within the class of graded skew Clifford algebras. In particular, one would like to determine whether the conditions on the matrices $M_\ell$ and the parameter matrix $\mu$ are not only sufficient but also close to being necessary in relevant families. This includes studying in detail those graded skew Clifford algebras that do not satisfy the hypotheses of Lemma~\ref{lemmaindepent} or Theorem~\ref{Theo1}, such as the algebra in Example~\ref{ExampleNoDS}, in order to decide whether they fail to be differentially smooth or whether a different choice of first-order calculus might still yield an integrable structure.

A second natural direction is to combine the techniques developed here with the rich classification theory of Artin-Schelter regular algebras arising from graded (skew) Clifford constructions. Many graded skew Clifford algebras appear as homogeneous coordinate rings of noncommutative quadrics or as building blocks in higher-dimensional regular algebras. It would be interesting to systematically compare differential smoothness with other notions of regularity (homological smoothness, twisted Calabi-Yau property, or the existence of balanced dualizing complexes) in these settings, and to identify geometric invariants (for instance, the point scheme or the behaviour of normalizing sequences) that control when a given Clifford-type algebra admits an integrable calculus.

From a more geometric viewpoint, one may also investigate additional structures compatible with the calculi constructed in this paper. Possible directions include the existence of hom-connections and connections with prescribed curvature, Hodge-type decompositions, or Laplace operators on graded skew Clifford algebras, in the spirit of noncommutative Riemannian geometry. Understanding when such extra structures can be defined, and how they behave under deformations of the matrices $(M_\ell)$ and the parameters $\mu_{ij}$, could shed further light on the interplay between the algebraic data and the resulting noncommutative geometry.


\begin{thebibliography}{60}


\bibitem{Apel1988} J. Apel, Gr\"obnerbasen in nichtkommutativen Algebren und ihre Anwendung, Doctoral Thesis, University of Leipzig, Leipzig, Germany (1988).





\bibitem{Bavula2020} V. V. Bavula, Generalized Weyl algebras and diskew polynomial rings, {\em J. Algebra Appl.} {\bf 19}(10) (2020) 2050194.

\bibitem{Bavula2023} Bavula, V. V. (2023). Description of bi-quadratic algebras on 3 generators with PBW basis. {\em J. Algebra} 631:695--730.

\bibitem{BavulaAlKhabyah2023} V. V. Bavula and A. Al Khabyah, Bi-quadratic algebras on 3 generators with PBW: class II.1 (2023) \url{https://arxiv.org/abs/2312.17182} 

\bibitem{BeggsBrzezinski2005} E. J. Beggs and T. Brzezi{\'n}ski, The Serre spectral sequence of a noncommutative fibration for de Rham cohomology, {\em Acta Math.} {\bf 195}(2) (2005) 155--196.



\bibitem{BellGoodearl1988} A. Bell and K. Goodearl, Uniform Rank Over Differential Operator Rings and Poincar{\'e}-Birkhoff-Witt Extensions, {\em Pacific J. Math.} {\bf 131}(1) (1988) 13--37.

\bibitem{BellSmith1990} A. D. Bell and S. P. Smith, Some 3-dimensional skew polynomial ring, University of Wisconsin, Milwaukee (1990).


\bibitem{Brzezinski2008} T. Brzezi{\'n}ski, Noncommutative Connections of The Second Kind, {\em J. Algebra Appl.} {\bf 7}(5) (2008) 557--573.

\bibitem{Brzezinski2011} T. Brzezi{\'n}ski, Divergences on Projective Modules and Noncommutative Integrals, {\em Int. J. Geom. Methods Mod. Phys.} {\bf 8}(4) (2011) 885--896.

\bibitem{Brzezinski2014} T. Brzezi{\'n}ski, On the Smoothness of the Noncommutative Pillow and Quantum Teardrops, {\em SIGMA Symmetry Integrability Geom. Methods Appl.} {\bf 10}(015) (2014) 1--8.

\bibitem{Brzezinski2015} T. Brzezi{\'n}ski, Differential smoothness of affine Hopf algebras of Gelfand-Kirillov of dimension two, {\em Colloq. Math.} {\bf 139}(1) (2015) 111--119.

\bibitem{Brzezinski2016}T. Brzezi\'nski, Noncommutative Differential Geometry of Generalized Weyl Algebras, {\em SIGMA Symmetry Integrability Geom. Methods Appl.} {\bf 12}(059) (2016) 1--18.

\bibitem{BrzezinskiElKaoutitLomp2010} T. Brzezi{\'n}ski, L. El. Kaoutit and C. Lomp, Noncommutative integral forms and twisted multi-derivations, {\em J. Noncommut. Geom.} {\bf 4}(2) (2010) 281--312.

\bibitem{BrzezinskiLomp2018} T. Brzezi{\'n}ski and C. Lomp, Differential smoothness of skew polynomial rings, {\em J. Pure Appl. Algebra} {\bf 222}(9) (2018) 2413--2426.

\bibitem{BrzezinskiSitarz2017}T. Brzezi{\'n}ski and A. Sitarz, Smooth geometry of the noncommutative pillow, cones and lens spaces, {\em J. Noncommut. Geom.} {\bf 11}(2) (2017) 413--449.

\bibitem{BuesoTorrecillasVerschoren2003}J. Bueso, J. G\'omez-Torrecillas and A. Verschoren, {\em Algorithmic Methods in noncommutative Algebra. Applications to Quantum Groups}, Mathematical Modelling: Theory and Applications (Springer Dordrecht, 2003).

\bibitem{Cassidy2010} T. Cassidy and M. Vancliff, Generalizations of Graded Clifford Algebras and of Complete Intersections, {\em J. Lond. Math. Soc.} {\bf 81}(1) (2010) 91--112.

\bibitem{Cassidy2014} T. Cassidy and M. Vancliff, Corrigendum: Generalizations of Graded Clifford Algebras and of Complete Intersections, {\em J. Lond. Math. Soc.} {\bf 90}(2) (2014) 631--636.



\bibitem{Connes1994} A. Connes, (1994). Noncommutative Geometry. New York: Academic Press.

\bibitem{CuntzQuillen1995} J. Cuntz and D. Quillen, (1995). Algebra Extensions and Nonsingularity. {\em J. Amer. Math. Soc.} 8(2):251--289.

\bibitem{DuboisViolette1988} M. Dubois-Violette, D\'erivations et calcul diff\'erentiel non commutatif, {\em C. R. Acad. Sci. Paris, Ser. I} {\bf 307} (1988) 403--408.

\bibitem{DuboisVioletteKernerMadore1990} M. Dubois-Violette, R. Kerner and J. Madore, Noncommutative differential geometry of matrix algebras, {\em J. Math. Phys.} {\bf 31}(2) (1990) 316--322.

\bibitem{Fajardoetal2020} W. Fajardo, C. Gallego, O.  Lezama, A. Reyes, H. Su\'arez, and H. Venegas, {\em Skew PBW Extensions: Ring and Module-theoretic properties, Matrix and Gr\"obner Methods, and Applications}, Algebra and Applications, Vol. 28 (Springer Cham, Switzerland, 2020).





\bibitem{GallegoLezama2010} C. Gallego and O. Lezama, Gr{\"o}bner {B}ases for {I}deals of $\sigma$-{P}{B}{W} {E}xtensions, {\em Comm. Algebra} {\bf 39}(1) (2010) 50--75.

\bibitem{GelfandKirillov1966} I. M. Gelfand and A. A. Kirillov, On fields connected with the enveloping algebras of {L}ie algebras, {\em Dokl. Akad. Nauk} {\bf 167}(3) (1966) 503--505.

\bibitem{GelfandKirillov1966b} I. M. Gelfand and A. A. Kirillov, Sur les corps li\'es aux algèbres enveloppantes des algèbres de Lie, {\em Publ. Math. IHES} {\bf 31} (1966) 5--19.

\bibitem{Giachettaetal2005} Giachetta, G., Mangiarotti, L., Sardanashvily, G. (2005). Geometric and {A}lgebraic {T}opological {M}ethods in {Q}uantum {M}echanics. Singapore: World Scientific Publishing.



\bibitem{Grothendieck1964} Grothendieck, A. (1964). \'El\'ements de g\'eom\'etrie alg\'ebrique. {\em Publ. Math. Inst. Hautes Etudes Sci.} 20:5--251.

\bibitem{Herrera2025} K. Herrera, S. Higuera, and A. Rubiano, A view toward the smooth geometry of Sklyanin algebras, (2025), \url{https://arxiv.org/pdf/2507.19719}

\bibitem{Hinchcliffe2005} O. Hinchcliffe, Diffusion algebras, PhD Thesis, University of Sheffield, Sheffield, England (2005).

\bibitem{IsaevPyatovRittenberg2001} A. P. Isaev, P. N. Pyatov and V. Rittenberg, Diffusion algebras, {\em J. Phys. A} {\bf 34}(29) (2001) 5815--5834.

\bibitem{IyuduShkarin2017} N. Iyudu and S. Shkarin, Three dimensional \uppercase{S}klyanin algebras and \uppercase{G}r{\"o}bner bases, {\em J. Algebra} {\bf 470} (2017) 379--419.

\bibitem{Karacuha2015} S. Kara\c cuha, Aspects of Noncommutative Differential Geometry, PhD Thesis, Universidade do Porto, Porto, Portugal (2015).

\bibitem{KaracuhaLomp2014}S. Kara\c cuha and C. Lomp, Integral calculus on quantum exterior algebras, {\em Int. J. Geom. Methods Mod. Phys.} {\bf 11}(04) (2014) 1450026.

\bibitem{Krahmer2012} Kr\"ahmer, U. (2012). On the Hochschild (co)homology of quantum homogeneous spaces. {\em Israel J. Math.} 189(1):237--266.

\bibitem{KrauseLenagan2000} G. R. Krause and T. H. Lenagan, {\em Growth of Algebras and Gelfand–Kirillov Dimension. Revised Edition}, Graduate Studies in Mathematics 22 (American Mathematical Society, 2000).



\bibitem{Lezama2014b} O. Lezama and A. Reyes, Some Homological Properties of Skew PBW Extensions, {\em Comm. Algebra} {\bf 42}(3) (2014) 1200--1230.







\bibitem{Levandovskyy2005} V. Levandovskyy, Non-commutative Computer Algebra for polynomial algebras: Gr\"obner bases, applications and implementation, Doctoral Thesis, Universit\"at Kaiserslautern, Kaiserslautern, Germany (2005).

\bibitem{Li2002} H. Li, {\em Noncommutative Gr\"obner Bases and Filtered-Graded Transfer}, Lecture Notes in Math. 1795 (Springer Berlin, Heidelberg, 2002).

\bibitem{Manin1997} Y. Manin, {\em Gauge Field Theory and Complex Geometry. Second Edition}, Grundlehren der mathematischen Wissenschaften, Vol. 289 (Springer Berlin, Heidelberg, 1997).

\bibitem{McConnell2001} J. McConnell, J. Robson and L. Small, {\em Noncommutative Noetherian Rings}, American Mathematical Society, Vol. 30, 2001.

\bibitem{Nafari2015} M. Nafari and M. Vancliff, Graded Skew Clifford Algebras that are twists of Graded Clifford Algebras, {\em Comm. Algebra} {\bf 43}(2) (2015) 719--725.

\bibitem{Nafari2011} M. Nafari, M. Vancliff and J. Zhang, Classifying quadratic quantum $\mathbb{P}^2$S by using graded skew Clifford algebras, {\em J. Algebra} {\bf 346}(1) (2011) 152--164.


\bibitem{PyatovTwarock2002} P. N. Pyatov and R. Twarock, Construction of diffusion algebras, {\em J. Math. Phys.} {\bf 43}(6) (2002) 3268--3279.

\bibitem{RedmanPhDThesis1996} I. T. Redman, The noncommutative algebraic geometry of some skew polynomial algebras, PhD Thesis, University of Wisconsin - Milwaukee, Wisconsin, United States of America (1996).

\bibitem{Redman1999} I. T. Redman, The homogenization of the three dimensional skew polynomial algebras of type I, {\em Comm. Algebra} {\bf 27}(11) (1999) 5587--5602.




\bibitem{ReyesSarmiento2022} A. Reyes and C. Sarmiento, On the differential smoothness of 3-dimensional skew polynomial algebras and diffusion algebras, {\em Internat. J. Algebra and Comput.} {\bf 32}(3) (2022) 529--559.


\bibitem{ReyesSuarez2017} A. Reyes and H. Su\'arez, $\sigma$-PBW Extensions of Skew Armendariz Rings, {\em Adv. Appl. Clifford Algebras} {\bf 27}(4) (2017) 3197--3224.

\bibitem{Rosenberg1995} A. L. Rosenberg, {\em Noncommutative Algebraic Geometry and Representations of Quantized Algebras}, Mathematics and Its Applications, Vol. 330 (Springer Dordrecht, 1995).








\bibitem{Schelter1986} Schelter, W. F. (1986). Smooth algebras. {\em J. Algebra} 103(2):677--685.

\bibitem{SeilerBook2010} W. M. Seiler, {\em Involution. {T}he {F}ormal {T}heory of {D}ifferential {E}quations and its {A}pplications in {C}omputer {A}lgebra}, Algorithms and Computation in Mathematics (AACIM), Vol. 24 (Springer Berlin, Heidelberg, 2010).






\bibitem{StaffordZhang1994} Stafford, J. T., Zhang, J. J. (1994). Homological Properties of (Graded) Noetherian PI Rings. {\em J. Algebra} 168(3):988--1026.






\bibitem{Vancliff2015a} M. Vancliff, The Interplay of Algebra and Geometry in the Setting of Regular Algebras, {\em Commut. Algebra and Noncommut. Algebra. Geom.} {\bf 6} (2015) 371--390.

\bibitem{Vancliff2015b} M. Vancliff, On the Notion of Complete Intersection Outside the Setting of Skew Polynomial Rings, {\em Comm. Algebra} {\bf 43}(2) (2015) 460--470.

\bibitem{VandenBergh1998}M. Van den Bergh, A relation between Hochschild homology and cohomology for Gorenstein rings. {\em Proc. Amer. Math. Soc.} 126(5):1345--1348. Erratum, Ibid. (1998). 130(9):2809--2810.

\bibitem{Veerapen2017} P. Veerapen, Graded Clifford Algebras and Graded Skew Clifford Algebras and Their Role in the Classification of Artin--Schelter Regular Algebras, {\em Adv. Appl. Clifford Algebr.} {\bf 27}(3) (2017) 2855--2871.



\bibitem{Woronowicz1987} S. L. Woronowicz, Twisted $SU(2)$ {G}roup. An {E}xample of a {N}oncommutative {D}ifferential {C}alculus, {\em Publ. Res. Inst. Math. Sci.} {\bf 23}(1) (1987) 117--181.


\end{thebibliography}
\end{document}